\documentclass[12pt]{amsart}

\usepackage{amsmath,amssymb,amsthm}
\usepackage[bookmarksopen=true,
            bookmarksnumbered=true,
            colorlinks=true,
            linkcolor=blue,
            pdftitle={The eventual shape of the Betti tables of powers of ideals},
            pdfsubject={Commutative Algebra},
            pdfkeywords={Asymptotic linearity, regularity, multigraded regularity, Betti number},
            pdfpagemode=UseOutlines,
            pdfproducer={Latex with hyperref},
            letterpaper=true,
            pdfcreator={latex->dvips->ps2pdf}]{hyperref}
\usepackage{graphicx,epsfig}
\usepackage{enumerate}
\usepackage{xypic}
  \input{xy}
  \xyoption{all}
\usepackage{url}

\numberwithin{equation}{section}

\theoremstyle{plain}
\newtheorem{theorem}{Theorem}[section]

\newtheorem{lemma}[theorem]{Lemma}

\newtheorem{corollary}[theorem]{Corollary}

\newtheorem{proposition}[theorem]{Proposition}

\theoremstyle{definition}
\newtheorem{example}[theorem]{Example}
\newtheorem{remark}[theorem]{Remark}
\newtheorem{definition}[theorem]{Definition}

\def\0{{\bf 0}}

\def\F{\mathbb{F}}
\def\G{\mathbb{G}}
\def\M{\mathbb{M}}

\def\NN{\mathbb{N}}
\def\R{\mathcal{R}}
\def\T{{\bf T}}
\def\ZZ{\mathbb{Z}}
\def\RR{\mathbb{R}}

\def\ra{\rightarrow}
\def\a{\frak{a}}
\def\b{\frak{b}}
\def\c{{\bf c}}
\def\x{{\bf x}}
\def\m{\frak{m}}
\def\nn{\frak{b}}

\DeclareMathOperator{\supp}{Supp}
\DeclareMathOperator{\reg}{reg}
\DeclareMathOperator{\tor}{Tor}
\DeclareMathOperator{\Hom}{Hom}

\begin{document}

\author{Amir Bagheri}
\address{Institut de Math\'ematiques de Jussieu\\
UPMC, Boite 247, 4, place Jussieu, F-75252 Paris Cedex, France}
\email{bagheri@math.jussieu.fr}

\author{Marc Chardin}
\address{Institut de Math\'ematiques de Jussieu\\
UPMC, Boite 247, 4, place Jussieu, F-75252 Paris Cedex, France\\
\url{http://people.math.jussieu.fr/~chardin}}
\email{chardin@math.jussieu.fr}

\author{Huy T{\`a}i H{\`a}}
\address{Department of Mathematics\\
Tulane University\\
6823 St. Charles Ave., New Orleans, LA 70118, USA\\
\url{www.math.tulane.edu/~tai/}}
\email{tha@tulane.edu}

\title{The eventual shape of Betti tables of powers of ideals}


\begin{abstract}
Let $G$ be an abelian group and $S$ be a $G$-graded a Noetherian algebra over a commutative ring $A\subseteq S_0$. Let $I_1, \dots, I_s$ be $G$-homogeneous ideals in $S$, and let $M$ be a finitely generated $G$-graded $S$-module. We show that the shape of nonzero $G$-graded Betti numbers of $MI_1^{t_1} \dots I_s^{t_s}$ exhibit an eventual linear behavior as the $t_i$s get large.
\end{abstract}

\maketitle


\section{Introduction}

It is a celebrated result (cf. \cite{CHT, Ko, TW}) that if $I \subseteq S$ is a
homogeneous ideal in a Noetherian standard $\NN$-graded algebra and $M$ is a finitely
generated $\ZZ$-graded $S$-module then the regularity $\reg(I^tM)$ is
asymptotically a linear function for $t \gg 0$. This asymptotic linear function
and the stabilization index have also been studied in \cite{Be, Ch, EH, EU, Ha}.

In the case $S$ is a polynomial ring over a field, a more precise result is proved in
 \cite{CHT}: the maximal degree of an $i$-th syzygy of $I^tM$ is eventually a linear
 function of $t$. Our interest here is to understand the eventual behavior of the
 degrees of all the minimal generators of the $i$-th syzygy module. This is of particular interest when the grading is given by
 some finitely generated abelian group $G$ that is not $\ZZ$, as in this case the result for
 the regularity of powers do not have an evident analogue. One can in particular
 consider the Cox ring of a toric variety, graded by the divisor class group.

We shall actually show, in the $G$-graded setting, that the collection of nonzero $G$-graded Betti numbers of
$MI_1^{t_1} \dots I_s^{t_s}$ exhibits an asymptotic linear behavior as the
$t_i$s get large.
Let us also point out that, even when an explicit minimal free resolution of these powers is known, for instance when $I$ is a complete intersection ideal, the degrees of $i$-th syzygies of $I^t$ do not exhibit trivially  a linear behavior.

Throughout the paper, let $G$ be a finitely generated abelian group and let $S = A[x_1, \dots, x_n]$ be a $G$-graded algebra over a commutative ring $A\subseteq S_0$. Hence $A = S/(x_1
, \dots, x_n)$ is a $G$-graded $S$-module concentrated in degree $0$.

Our work hinges on the relationship between multigraded Betti numbers and the graded pieces of $\tor^S_i(MI_1^{t_1} \dots I_s^{t_s}, A)$; we, thus, examine the support of $\tor^S_i(MI_1^{t_1}\cdots I_s^{t_s}, A)$ as the $t_i$s get large. Our main result, in the case of a single ideal, gives the following:

\begin{theorem}[see Theorem \ref{thm.arbitrarydegree2}] \label{thm.intro1}
Let $I=(f_{1}, \dots, f_{r})$ be a $G$-homogeneous $S$-ideal, and let $M$ be a finitely generated $G$-graded $S$-module. Set $\Gamma := \{\deg_G(f_{i})\}_{i=1}^{r}$.

Let $\ell \ge 0$, and assume that $\ell =0$ or $A$ is Noetherian.

There exist a finite collection of elements $\delta_i \in G$, a finite collection of integers $t_{i}$, and a finite collection of non-empty tuples $E_{i}=(\gamma_{i,1},\ldots ,\gamma_{i,s_i})$ of elements in $\Gamma$, such that the
elements $(\gamma_{i,j+1}-\gamma_{i,j}, \  1\leq j\leq s_i -1)$ are linearly independent, satisfying

$$\supp_G(\tor^S_\ell(MI^{t},A)) = \bigcup_{i=1}^m \big( \bigcup_{{c_1+\cdots +c_{s_i} = t-t_{i},}\atop{\ c_1,\ldots ,c_{s_i}\in \NN}} c_1\gamma_{i,1}+\cdots +c_{s_i}\gamma_{i,s_i} +\delta_i \big), \quad  \forall t\geq \max_i \{ t_{i}\} .
$$
\end{theorem}

The condition of linear independence stated for $E_i$ implies that $c_1\gamma_{i,1}+\cdots +c_{s_i}\gamma_{i,s_i}\not= c'_1\gamma_{i,1}+\cdots +c'_{s_i}\gamma_{i,s_i}$ if $c_1+\cdots +c_{s_i} =c'_1+\cdots +c'_{s_i} $ and $(c_1,\ldots ,c_{s_i})\not= (c'_1,\ldots ,c'_{s_i})$. Notice the important fact that the elements in $E_i$ are all in $\Gamma$, the set of degrees of generators of $I$.

Theorem \ref{thm.arbitrarydegree2} is proved in the last section of the paper. Our proof is based on the two following observations. Firstly, the multi-Rees module (sometimes also referred to as the Rees modification) $M\R = \bigoplus_{t_i \ge 0}MI_1^{t_1} \dots I_s^{t_s}$ is a $G \times \ZZ^s$-graded module over the multi-Rees algebra $\R = \bigoplus_{t_i \ge 0}I_1^{t_1} \dots I_s^{t_s}$. Secondly, if $G'$ is a finitely generated abelian group and $R = S[T_1, \dots, T_r]$ is a $G \times G'$-graded polynomial extension of $S$, such that $\deg_{G \times G'}(a) = (\deg_G(a), 0)$ for any $a \in S$, then for a graded complex $\G_\bullet$ of free $R$-modules,
$$H_i\big((\G_\bullet)_{G \times \{\delta\}} \otimes_S A\big) = H_i\big(\G_\bullet \otimes_S A\big)_{G \times \{\delta\}},$$
where $(\bullet)_{G \times \{\delta\}}$ denotes the degree $G \times \{\delta\}$-strand of the corresponding complex. In particular, if $\G_\bullet$ is a $G \times G'$-graded free resolution of an $R$-module $N$, as $(\G_\bullet)_{G \times \{\delta\}}$ is a $G$-graded free resolution of the $S$-module $N_{G \times \{\delta\}}$, it follows that
$$\tor^S_i(N_{G \times \{\delta\}}, A) = H_i(\G_\bullet \otimes_S A)_{G \times \{\delta\}},$$
in which $\G_\bullet \otimes_S A$ is viewed as a $G \times G'$-graded complex of free $A[T_1, \dots, T_r]$-modules. These observations allow us to bring the problem to studying the support of graded $A[T_1, \dots, T_r]$-modules. We proceed by making use of the notion of initial submodules to reduce to the case when the module is a quotient ring obtained by a monomial ideal. The result then follows by considering the Stanley decomposition of such a quotient ring.

In its full generality, our proof of Theorem \ref{thm.arbitrarydegree2} is quite technical, so we start in Section \ref{sec.samedegree} by considering first the case when each $I_i$ is equi-generated. In fact, in this case, with the additional assumption that $A$ is a Noetherian ring, our results are much stronger; the asymptotic linearity appears clearer and the proofs are simpler. We can also examine the support of local cohomology modules of $MI_1^{t_1} \dots I_s^{t_s}$. Our results, in the case when each $I_i$ is finitely generated in a single degree $\gamma_i$, are stated as follows.

\begin{theorem}[Theorem \ref{thm.Betti}] \label{thm.intro2}
Assume that $i=0$ or $A$ is Noetherian. Then there exists a finite set $\Delta_i \subseteq G$ such that
\begin{enumerate}
\item For all $t = (t_1, \dots, t_s) \in \NN^s$, $\tor^S_i(MI_1^{t_1}\cdots I_s^{t_s},A)_\eta = 0$  if $\eta \not\in \Delta_i + t_1\gamma_1 +\cdots +t_s\gamma_s$.
\item There exists a subset $\Delta_i' \subseteq \Delta_i$ such that $\tor^S_i(MI_1^{t_1}\cdots I_s^{t_s},A)_{\eta + t_1\gamma_1 +\cdots +t_s\gamma_s}\not= 0$ for $t\gg 0$ and $\eta \in \Delta_i' $, and
$\tor^S_i(MI_1^{t_1}\cdots I_s^{t_s},A)_{\eta + t_1\gamma_1 +\cdots +t_s\gamma_s} = 0$ for $t\gg 0$ and $\eta \not\in \Delta_i' $.
\item Let $A\rightarrow k$ be a ring homomorphism to a field $k$. Then for any $\delta$ and any $j$, the function
$$
\dim_{k} \tor^A_j( \tor^S_i(MI_1^{t_1}\cdots I_s^{t_s},A)_{\delta + t_1\gamma_1 +\cdots +t_s\gamma_s},k)
$$ is polynomial in the $t_i$s for $t \gg 0$, and the function
$$
\dim_{k} \tor^S_i(MI_1^{t_1}\cdots I_s^{t_s},k)_{\delta + t_1\gamma_1 +\cdots +t_s\gamma_s}
$$
 is polynomial in the $t_i$s for $t \gg 0$.
\end{enumerate}
\end{theorem}

\begin{theorem}[Theorem \ref{thm.localcohomology}] \label{thm.intro3}
Let $\b$ be a $G$-homogeneous ideal in $S$ such that for any $i \ge 0$ and $\delta \in G$, $H^i_\b(S)_\delta$ is a finitely generated $A$-module. Then, if $A$ is Noetherian,  for $i \ge 0$, there exists a subset $\Lambda_i \subseteq G$ such that
\begin{enumerate}
\item $H^i_\b(MI_1^{t_1}\cdots I_s^{t_s})_{\eta  + t_1\gamma_1 +\cdots +t_s\gamma_s} \not= 0$ for $t = (t_1, \dots, t_s) \gg 0$ if only if $\eta \in \Lambda_i$.
\item Let $A\rightarrow k$ be a ring homomorphism to a field $k$. Then for any $\delta$ and any $j$, the function
$$
\dim_{k} \tor^A_j(  H^i_\b(MI_1^{t_1}\cdots I_s^{t_s})_{\delta+t_1\gamma_1 +\cdots +t_s\gamma_s},k)
$$
is a polynomial in the $t_i$s for $t \gg 0$.
\end{enumerate}
\end{theorem}

In the simplest scenario, when $S$ is a standard graded polynomial ring over a field $k$, $\m \subseteq S$ the homogeneous maximal $S$-ideal, $s=1$, and $I \subseteq S$ is a homogeneous ideal generated in degree $d$, Theorems \ref{thm.intro2} and \ref{thm.intro3} give the following result.

\begin{corollary} \label{cor.intro4}
For $i \ge 0$, there exist $t_i$   and finite sets $\Delta_i' \subseteq \Delta_i \subseteq \ZZ$ such that
\begin{enumerate}
\item for all $t \in \NN$, $\tor^S_i(MI^t, k)_{\eta +td} = 0$ if $\eta \not\in \Delta_i$;
\item for $t \geq t_i$, $\tor^S_i(MI^t,k)_{\eta+td} \not= 0$ if and only if $\eta \in \Delta_i'$;
\item for any $\eta \in \ZZ$, the function $\dim_k \tor^S_i(MI^t,k)_{\eta + td}$ is a polynomial in $t$ for $t \geq t_i$;
\item for any $\theta \in \ZZ$, the function $\dim_k H^i_\m (MI^t)_{\theta+td}$ is a polynomial in $t$, for $t \gg 0$.
\end{enumerate}
\end{corollary}

While writing this paper, we were informed by Whieldon that in her recent work \cite{Wh}, a similar result to Corollary \ref{cor.intro4} (1)--(3) is proved.
We later learned that Pooja Singla  also proved these results, in the second chapter of her thesis \cite{Si}, independently. She also shows in
\cite{Si} that if $I$ is graded ideal, then for any $a,b\in\ZZ$,
 $\dim_k \tor^S_i(I^t,k)_{a+bt}$ is a quasi-polynomial in $t$ for $t\gg 0$, and give results describing the regularity of $I_1^{t_1}\cdots
 I_s^{t_s}$ for $t\gg 0$.

In the general setting, when $A$ is not necessarily a field and $S$ is not necessarily standard graded, the problem is more subtle, and requires more work and a different approach.

We choose not to restrict to a Noetherian base ring in all our results as it seemed to us that it makes the presentation more clear to put this hypothesis only in the statements where it is of use in the proof.

\noindent{\bf Acknowledgement.} Our work was done when the third named author
visited the other authors at the Institut de Math\'ematiques de Jussieu, UPMC.
The authors would like to thank the Institut for its support and hospitality.
The third named author is partially supported by NSA grant H98230-11-1-0165.


\section{Preliminaries} \label{sec.prel}

In this section, we collect necessary notations and terminology used in the paper, and prove a few auxiliary results. For basic facts in commutative algebra, we refer the reader to \cite{Ei,Mat}.

From now on, $G$ denotes a finitely generated abelian group, $S = A[x_1, \dots, x_n]$ is a $G$-graded algebra over a commutative ring with identity,  $A\subseteq S_0$ and $M$ is a $G$-graded $S$-module. By abusing notation, we shall use $0$ to denote the identity of all abelian groups considered in the paper; the particular group will be understood from the context of its use.

\begin{definition} Let $E \subseteq G$ be a collection of elements in $G$. We say that $E$ is a {\it linearly independent} subset of $G$ if $E$ forms a basis for a free submonoid of $G$.
\end{definition}

\begin{definition}
The {\it support} of $M$ in $G$ is defined to be
$$\supp_G(M) = \{ \gamma \in G ~\big|~ M_\gamma \not= 0\}.$$
\end{definition}

\begin{remark}
When $A$ is a field, let $\F_\bullet$ be a minimal $G$-graded free resolution of $M$ over $S$, where
$$\F_i = \bigoplus_{\eta \in G} S(-\eta)^{\beta^i_\eta(M)}.$$
The numbers $\beta^i_\eta(M)$ are called the {\it multigraded (or $G$-graded) Betti numbers} of $M$ and
$$\beta^i_\eta(M) = \dim_A \tor^S_i(M,A)_\eta$$ as, by definition, the maps in $\F_\bullet\otimes_S A$ are zero maps.
\end{remark}

More generally, we shall prove the following lemma relating the multigraded Betti numbers of $M$ and the support of $\tor^S_*(M,A)$.

\begin{lemma}
Let $\F_\bullet$ be a $G$-graded free resolution of a $G$-graded $S$-module $M$. Then

(1)  $\F_i$ has a summand $S(-\gamma)$ for any $\gamma \in \supp_G (\tor^S_i(M,A))$.

(2)  Assume that there exists $\phi\in\Hom_{\ZZ}(G,\RR )$ such that $\phi (\deg (x_i))>0$ for all $i$ and
$M$ is finitely generated. Then there exists a $G$-graded free resolution $\F'_\bullet$ of $M$ such that
$$
\F'_i=\bigoplus_{\ell\in E'_i}S(-\gamma_{i,\ell} )\quad \hbox{with}\quad \gamma_{i,\ell} \in \bigcup_{j\leq i}\supp_G (\tor^S_j(M,A)),\ \forall \ell .
$$
\end{lemma}

Notice that, without further restrictions
on $A$ and/or $M$ one cannot in general choose $\F'_i$ so that $\gamma_{i,\ell} \in \supp_G (\tor^S_i(M,A)),\ \forall \ell $.

\begin{proof}
For (1), let $K$ be defined by the exact sequence
$$
0\ra K\ra \F_0 \ra M\ra 0
$$
and notice that $\F_0 \otimes_S A\ra M\otimes_S A$ is onto. As $(\F_0 \otimes_S A)_\gamma \not=0$ if
and only if
$S(-\gamma )$ is a direct summand of $\F_0$, the result holds for $i=0$. Furthermore, $\cdots \ra \F_1\ra K\ra 0$
is a resolution of $K$, $\tor_1^S(M,A)$ is a graded submodule of $K\otimes_S A$ and $\tor_i^S(M,A)\simeq \tor_{i-1}^S(K,A)$ for
$i\geq 2$, which implies the result by induction on $i$.

To prove the second statement, we relax the finite generation of $M$ by the following condition, which will enable us
to make induction :  $\exists q\in \RR,\ \phi (\deg a)>q,\ \forall a\in M$. Notice that if a module satisfies this condition,
any of its submodules satisfies the same condition.

Set $T_j:=\supp_G (\tor^S_j(M,A))$. Let  $\psi :\F_0 =\bigoplus_{\ell\in E_0}S(-\gamma_{0,\ell} )\ra M$ be the augmentation and  $E'_0:=\{ \ell\in E_0\ \vert\ \gamma_{0,\ell}\in T_0\}$. Denote by
 $\psi '$  the restriction of $\psi$ to $\F'_0:=\bigoplus_{\ell\in E'_0}S(-\gamma_{0,\ell} )$.

We now prove that $\psi '$ is onto. First notice that $\psi'\otimes_S A$ is surjective. Assume that $\psi'$ is not surjective, let $M'$ be the image of $\psi'$ and
 $$
 h:=\inf\{ \phi (\gamma), M_\gamma \not= M'_\gamma \}\geq \inf\{ \phi (\deg (a)),\ a\in M\} \in \RR.
 $$


Set $\epsilon :=\min\{ \phi (\deg (x_i)),\ i=1,\ldots ,n\} >0$ and choose $m\in M_\nu \setminus M'_\nu$ for
some $\nu$ with $h\leq \phi (\nu )<h+\epsilon$. As  $\psi'\otimes_S A$ is onto, there exists $m'\in M'$ such that $m$ is of the form $m'+\sum_{i=1}^n m_ix_i$. Now, for some value $i$ we have $m_i\not\in M'$. It then follows that $\phi(\deg (m_i)) =  \phi(\nu) - \phi (\deg(x_i)) < h$, contradicting the definition of $h$.

We will now prove (2) by induction on $i$. To end this, assume that there exists a graded complex
$$
0\ra \F'_i\ra \cdots \ra \F'_0\ra M\ra 0
$$
with at most non zero homology in homological degree $i\geq 0$, and such that $\F'_i=\bigoplus_{\ell\in E'_i}S(-\gamma_{i,\ell} )$ with $\gamma_{i,\ell} \in \bigcup_{j\leq i}T_j$.

If the complex is exact, our claim is proved; otherwise, by setting $K_i\subset  \F'_i$ for the $i$-th homology module
of the complex and $Q_i:=\ker (\F'_i\otimes_S A \ra \F'_{i-1}\otimes_S A)$, one has
$$
\supp_G (K_i \otimes_S A) =\supp_G ( \tor_{i+1}^S (M, A))\cup \supp_G (\ker (Q_i\ra  \tor_{i}^S (M, A))).
$$

Applying the argument above to a graded onto map $\F \ra K_i$, and using that by induction
$$
\supp_G (K_i \otimes_S A) \subseteq T_{i+1}\cup \supp_G (\F'_{i}\otimes_S A)\subseteq \cup_{j\leq i+1}T_{j}
$$
one obtains a graded free $S$-module $\F'_{i+1}$
as claimed mapping onto $K_i$.
\end{proof}

Let $t = (t_1, \dots, t_s) \in \ZZ^s$. We shall write $t \ge 0$ (respectively, $t > 0$, $t \le 0$ and $t \not= 0$) if $t_i \ge 0$ (respectively, $t_i > 0, t_i \le 0$ and $t_i \not= 0$) for all $i = 1, \dots, s$. For a property that depends on a parameter $t\in \ZZ^s$, one says that {\it the property holds for $t\gg 0$} if there exists $t_0\in \ZZ^s$ such that it holds for $t\in t_0+\NN^s$. The following semi-classical lemma will be of use.

\begin{lemma}\label{tame}
Let $R$ be a finitely generated $\NN^s$-graded algebra over a commutative ring $A$. Let $M$ be a finitely generated $\ZZ^s$-graded $R$-module. Then either $M_t=0$ for $t\gg 0$ or $M_t\not= 0$ for $t\gg 0$.
\end{lemma}

\begin{proof}
Let $\nn = \bigoplus_{t_i \ge 1} R_{(t_1, \dots, t_s)}$ be the ideal generated by elements of strictly positive degrees.
If $M\not= H^0_\nn (M)$ one has $M_t\not= 0$ for any $t\in t_0+\NN^s$, where $t_0$ is the degree
of a non zero element in $M/H^0_\nn (M)$.

If $M=H^0_\nn (M)$ then any generator $a$ of $M$ spans a submodule of $M$ that is zero
in degrees $\deg (a)+b_a (1,\ldots ,1)+{\NN^s}$, where $b_a$ is such that
for any product $p$ of $b_a$ elements among  the finitely many generators of the $R$-ideal $\nn$,
one has $pa=0$.
As $M$ is finitely generated, the result follows.
\end{proof}

We shall make use of the notion of {\it initial modules} with respect to a monomial order. This is a natural extension of the familiar notion of initial ideals in a polynomial ring. Let $F$ be a free $S$-module. We can write $F = \bigoplus_{i \in I}Se_i$. A {\it monomial} in $F$ is of the form $\x^{\alpha}e_i$, where $\x^\alpha$ is a monomial in $S$ and $i \in I$. A {\it monomial order} in $F$ is a total order, say $\prec$, on the monomials of $F$ satisfying the following condition: \newline
\hspace*{6ex}
if $u \prec v$ and $w\not= 1$ is a monomial in $S$ then $u\prec uw \prec vw$.

It can be seen that $\prec$ is a well ordering, i.e., every non-empty subset of the monomials in $F$ has a minimal element.
We refer the reader to \cite[15.2]{Ei} for more details on monomial orders on free modules.

\begin{definition}
Let $F$ be a free $S$-module, and let $K$ be an $S$-submodule of $F$. Let $\prec$ be a monomial order in $F$. The {\it initial module} of $K$, denoted by $\operatorname{in}_\prec(K)$, is defined to be the $S$-submodule of $F$ generated by
$$\{\x^{\alpha} e_i ~\big|~ \exists f = \x^{\alpha} e_i + \text{ smaller terms } \in K\}.$$
\end{definition}

\begin{proposition} \label{lem.initial}
Let $F$ be a free $G$-graded $S$-module, and let $K$ be a $G$-graded $S$-submodule of $F$. Let $\prec$ be a monomial order in $F$. Then $\operatorname{in}_\prec(K)$ is a $G$-graded $S$-module of $F$, and
$$\supp_G(F/K) = \supp_G(F/\operatorname{in}_\prec(K)).$$
\end{proposition}

\begin{proof} It is clear from the definition that $\operatorname{in}_\prec(K)$ is a $G$-graded $S$-module. To prove the proposition, we need to show that for any $\mu \in G$, $K_\mu = F_\mu$ if and only if $\operatorname{in}_\prec(K)_\mu = F_\mu$.

Clearly, if $K_\mu = F_\mu$ then all monomials of degree $\mu$ in $F$ are elements of $K$, and thus, are elements of $\operatorname{in}_\prec(K)$. Therefore, in this case, $\operatorname{in}_\prec(K)_\mu = F_\mu.$

Suppose now that $\operatorname{in}_\prec(K)_\mu = F_\mu.$ Let $\x^{\alpha}e_i$ be the smallest monomial of degree $\mu$ in $F$ but not in $K$, if it exists. Then $\x^{\alpha}e_i \in \operatorname{in}_\prec(K)_\mu$. Thus, there exists an element $f \in K$ of the form $$f = \x^{\alpha}e_i + g,$$
where $g$ consists of monomials that are smaller than $\x^{\alpha}e_i$ with respect to $\prec$. Since $K$ is a $G$-graded $S$-module, we can choose $f$ to be $G$-homogeneous of degree $\mu$. That is, all its monomials are of degree $\mu$. This implies that all monomials, and thus, all terms in $g$ are elements in $K$. In particular, $g \in K$. Therefore, $\x^{\alpha}e_i = f-g \in K$, a contradiction. Hence, $F_\mu = K_\mu$. The proposition is proved.
\end{proof}

One of our techniques is to take the collection of elements of certain degree from a complex. This construction gives what we shall call {\it strands} of the complex.

\begin{definition} Let $\F_\bullet$ be a $G$-graded complex of $S$-modules and let $\Gamma \subseteq G$. The {\it $\Gamma$-strand} of $\F_\bullet$, often denoted by $(\F_\bullet)_\Gamma$, is obtained by taking elements of degrees belonging to $\Gamma$ in $\F_\bullet$ and the boundary maps between these elements (since the complex is graded, the boundary maps are of degree 0). In particular, if $F = \bigoplus_{\gamma \in G} F_\gamma$ is a $G$-graded $S$-module, then $F_\Gamma := \bigoplus_{\gamma \in \Gamma} F_\gamma$.
Note that the degree $\Gamma$-strand of a complex/module is not necessarily a complex/module over $S$.
\end{definition}


\section{Forms of the same degree} \label{sec.samedegree}

In this section, we consider the case when every ideal $I_i$ is generated in a single degree. That is, when $\deg_G(f_{i,j}) = \gamma_i$ for all $j$.
We keep the notations of Section \ref{sec.prel}.

Let $G'$ denote a finitely generated abelian group. The following result, with $G'=\ZZ^s$, will be a key ingredient of our proof.

\begin{theorem} \label{thm.samedegree}
Let $R=S[T_1,\ldots ,T_r]$ be a $G\times G'$-graded polynomial extension of $S$ with $\deg_{G \times G'}(a) \in G\times 0$ for all $a\in S$ and $\deg_{G \times G'}(T_j)\in 0\times G'$ for all $j$.
Let $\M$ be a finitely generated $G \times G'$-graded $R$-module and let $i$ be an integer. Assume that $i=0$ or $A$ is a Noetherian ring. Then
\begin{enumerate}
\item There exists a finite subset $\Delta_i \subseteq G$ such that,
for any $t$,  $\tor^S_i(\M_{(*,t)}, A)_\delta = 0$ for all $\delta \not\in \Delta_i$.
\item Assume that $G' = \ZZ^s$. For $\delta \in \Delta_i$, $\tor^S_i(\M_{(*,t)}, A)_\delta = 0$ for $t\gg 0$ or $\tor^S_i(\M_{(*,t)}, A)_\delta \not=  0$ for $t\gg 0$.
If, furthermore, $A \rightarrow k$ is a ring homomorphism to a field $k$, then for any $j$, the function
$$
\dim_{k} \tor^A_j( \tor^S_i(\M_{(*,t)},A)_\delta, k)
$$ is polynomial in the $t_i$s for $t \gg 0$, and the function
$$
\dim_{k} \tor^S_i(\M_{(*,t)},k)_\delta
$$
 is polynomial in the $t_i$s for $t \gg 0$.
\end{enumerate}
\end{theorem}

\begin{proof} Let $\F_\bullet$ be a  graded free resolution of $\M$ over $R$, where $\F_i = \bigoplus_{\eta, j} R(-\eta, -j)^{\beta^i_{\eta,j}}$ is of finite rank for $i=0$, and for any $i$ when $A$ is Noetherian. For  $t\in G'$, the $(*,t)$-strand of $\F_\bullet$, denoted by $\F_\bullet^t$, is a $G$-graded free resolution of $\M_{(*,t)}$ over $S = R_{(*,0)}$, that is not necessarily minimal. Its $i$-th term is
$$\F_i^t = \bigoplus_{\eta,j}S(-\eta)^{\beta^i_{\eta,j}} \otimes_A B_{t-j},$$
where $B = A[T_1, \dots, T_r]$.

Let $\Delta_i = \{\eta ~\big|~ \exists j : \beta^i_{\eta,j} \not= 0\}.$ The module $\tor^S_i(\M_{(*,t)},A) = H_i(\F_\bullet^t \otimes_S A)$ is a subquotient of the module $\bigoplus_{\eta,j}A(-\eta)^{\beta^i_{\eta,j}} \otimes_A B_{t-j}$, and (1) is proved.

To prove (2), observe first that $\tor^S_i(\M_{(*,t)},A)_\delta = H_i(\F_\bullet^t \otimes_S A)_\delta$ and  $\big(A(-\eta) \otimes_A B_{t-j}\big)_\delta = A_{\delta - \eta} \otimes_A B_{t-j}$ is zero if $\eta \not= \delta$. Thus, $H_i(\F_\bullet^t \otimes_S A)_\delta$ is equal to $H_i(\F_\bullet^{[\delta]} \otimes_S A)_t$, where $\F_\bullet^{[\delta]}$ is the subcomplex of $\F_\bullet$ given by
$$\F_i^{[\delta]} = \bigoplus_{j} R(-\delta, -j)^{\beta^i_{\delta,j}} = \bigoplus_{j}[S(-\delta) \otimes_A B(-j)]^{\beta^i_{\delta,j}}.$$

As $\F_\bullet^{[\delta]} \otimes_S A$ is a graded complex of finitely generated $B$-modules, $H_i(\F_\bullet^{[\delta]} \otimes_S A)$ is a finitely generated $B$-module for any $i$ when $A$ is Noetherian.   Similarly, $\tor^S_i(\M_{(*,t)},k)_\delta = H_i(\F_\bullet^t \otimes_S k)_\delta=H_i(\F_\bullet^{[\delta]} \otimes_S k)$  is a finitely generated $k[T_1, \dots, T_r]$-module for any $i$ when $A$ is Noetherian.

This proves (2) in view of Lemma \ref{tame} and \cite[Theorem 1]{Ko1}.
\end{proof}

\begin{remark}
In the context of point (2) above, there is a graded spectral sequence of graded $B$-modules with second term $E^2_{j,i}(t)=\tor^A_j( \tor^S_i(\M_{(*,t)},A)_\delta, k)$ that converges
to $ \tor^S_{i+j}(\M_{(*,t)},k)_\delta$. It follows that all terms $E^p_{j,i}(t)$ for $p\geq 2$ are finitely generated $k[T_1, \dots, T_r]$-modules.
In particular, one can write :
$$
\dim_{k} \tor^S_\ell (\M_{(*,t)},k)_\delta=\sum_{i+j=\ell} \dim_{k} E^\infty_{j,i}(t)\leq \sum_{i+j=\ell} \dim_{k} E^p_{j,i}(t), \quad \forall p\geq 2,
$$
which provides a control on the Hilbert function (and polynomial) of $ \tor^S_\ell (\M ,k)_\delta$ in terms of the ones of $ \tor^A_j( \tor^S_i(\M ,A)_\delta, k)$ for $i+j=\ell$.
\end{remark}

We are now ready to examine the asymptotic linear behavior of nonzero $G$-graded Betti numbers and non-vanishing degrees of local cohomology modules of $MI_1^{t_1}\cdots I_s^{t_s}$.

In the next two theorems, let $R = S[T_{i,j} ~|~ 1 \le i \le s, 1 \le j \le r_i]$. Then $R$ is equipped with a $G \times \ZZ^s$-graded structure in which $\deg_{G \times \ZZ^s}(x_i) = (\deg_G(x_i), 0)$ and $\deg_{G \times \ZZ^s}(T_{i,j}) = (0,e_i)$, where $e_i$ is the $i$-th element in the canonical basis of $\ZZ^s$. Recall that $\gamma_i = \deg_G(f_{i,j})$ and let $\gamma := (\gamma_1, \dots, \gamma_s)$. For $t = (t_1, \dots, t_s) \in \ZZ^s$, let $I^t := I_1^{t_1} \dots I_s^{t_s}, T^t := T_1^{t_1} \dots T_s^{t_s}, I^tT^t(\gamma . t) := I_1^{t_1}(t_1\gamma_1) T_1^{t_1} \dots I_s^{t_s}(t_s \gamma_s) T_s^{t_s}$ and $MI^tT^t(\gamma . t) := M I_1^{t_1}(t_1\gamma_1)T_1^{t_1} \dots I_s^{t_s}(t_s\gamma_s) T_s^{t_s}.$

\begin{theorem} \label{thm.Betti}
Assume that $i=0$ or $A$ is Noetherian. Then there exists a finite set $\Delta_i \subseteq G$ such that
\begin{enumerate}
\item For all $t \in \NN^s$, $\tor^S_i(MI_1^{t_1}\cdots I_s^{t_s},A)_\eta = 0$  if $\eta \not\in \Delta_i + t_1\gamma_1 +\cdots +t_s\gamma_s$.
\item There exists a subset $\Delta_i' \subseteq \Delta_i$ such that $\tor^S_i(MI_1^{t_1}\cdots I_s^{t_s},A)_{\eta + t_1\gamma_1 +\cdots +t_s\gamma_s}\not= 0$ for $t\gg 0$ and $\eta \in \Delta_i' $ and
$\tor^S_i(MI_1^{t_1}\cdots I_s^{t_s},A)_{\eta + t_1\gamma_1 +\cdots +t_s\gamma_s} = 0$ for $t\gg 0$ and $\eta \not\in \Delta_i' $.
\item  If, furthermore, $A \rightarrow k$ is a ring homomorphism to a field $k$, then for any $\delta$ and any $j$, the function
$$
\dim_{k} \tor^A_j( \tor^S_i(MI_1^{t_1}\cdots I_s^{t_s},A)_{\delta + t_1\gamma_1 +\cdots +t_s\gamma_s},k)
$$ is polynomial in the $t_i$s for $t \gg 0$, and the function
$$
\dim_{k} \tor^S_i(MI_1^{t_1}\cdots I_s^{t_s},k)_{\delta + t_1\gamma_1 +\cdots +t_s\gamma_s}
$$
 is polynomial in the $t_i$s for $t \gg 0$.
\end{enumerate}
\end{theorem}

\begin{proof} Let $\R := \bigoplus_{t \ge 0}I^{t}T^{t}(\gamma .t)$ and $M\R := \bigoplus_{t \ge 0}MI^{t}T^{t}(\gamma .t)$ denote the (shifted) multi-Rees algebra and the multi-Rees module with respect to $I_1, \dots, I_s$, and $M$. The natural surjective map $\phi: R \twoheadrightarrow \R$ that sends $x_i$ to $x_i$ and $T_{i,j}$ to $f_{i,j}T_i$ makes $M\R$ a finitely generated $G \times \ZZ^s$-graded module over $R$.

Observe that, for any $\delta \in G$ and $t \in \ZZ^s$, $M\R_{(\delta,t)} \simeq [MI^t(\gamma .t)]_\delta = [MI^t]_{\delta+\gamma .t}$, where in the last term $\delta + \gamma .t := \delta + t_1\gamma_1 + \dots + t_s\gamma_s$. Thus, the assertion follows by applying Theorem \ref{thm.samedegree} to the $R$-module $\M :=M\R$.
\end{proof}

\begin{theorem} \label{thm.localcohomology}
Let $\b$ be a $G$-homogeneous ideal in $S$ such that for any $i \ge 0$ and $\delta \in G$, $H^i_\b(S)_\delta$ is a finitely generated $A$-module. Then, if $A$ is Noetherian,  for $i \ge 0$, there exists a subset $\Lambda_i \subseteq G$ such that
\begin{enumerate}
\item $H^i_\b(MI_1^{t_1}\cdots I_s^{t_s})_{\eta  + t_1\gamma_1 +\cdots +t_s\gamma_s} \not= 0$ for $t = (t_1, \dots, t_s) \gg 0$ if only if $\eta \in \Lambda_i$.
\item  If, furthermore, $A \rightarrow k$ is a ring homomorphism to a field $k$, then for any $\delta$ and any $j$, the function
$$
\dim_{k} \tor^A_j(  H^i_\b(MI_1^{t_1}\cdots I_s^{t_s})_{\delta+t_1\gamma_1 +\cdots +t_s\gamma_s},k)
$$
is a polynomial in the $t_i$s for $t \gg 0$.
\end{enumerate}
\end{theorem}

\begin{proof} Since taking local cohomology respects the $G$-homogeneous degree, we have
$$H^i_\b(MI^t)_{\delta+\gamma .t} = H^i_{\b R}(M\R)_{(\delta, t)} = \big(H^i_{\b R}(M\R)_{(\delta,*)}\big)_t.$$

Let $B = A[T_{i,j}, 1\leq i\leq s,\; 1\leq j\leq r_i]$. Since $B$ is a flat extension of $A$, $H^i_{\b R}(R)_{(\delta,*)} = H^i_{\b R}(B \otimes_A S)_{(\delta,*)}$ is a finitely generated $B$-module. Let $\F_\bullet$ be the minimal $G \times \ZZ^s$-graded free resolution of $M\R$ over $R$. Since $A$ is Noetherian,  each term $\F_j$ of $\F_\bullet$ is of finite rank. This implies that for all $\delta \in G$, $i \ge 0$ and $j \ge 0$, $H^i_{\b R}(\F_j)_{(\delta,*)}$ is a finitely generated $B$-module. The spectral sequence $H^i_{\b R}(\F_j )\Rightarrow H^{i-j}_{\b R}(M\R )$
implies that $H^i_{\b R}(M\R)_{(\delta,*)}$ is a finitely generated multigraded $B$-module.  This proves (2) in view of  \cite[Theorem 1]{Ko1}.

Notice that $H^i_{\b R}(M\R)_{(\delta,t)} = 0$ for all $t \gg 0$ if and only if $K_\delta = H^i_{\b R}(M\R)_{(\delta,*)}$ is annihilated by a power of the ideal $\a :=\bigcap_{i=1}^s (T_{i,1}, \dots, T_{i,r_i})$. Hence (1) holds with
$$\Lambda_i := \{\delta \in G ~\big|~ K_\delta \not= H^0_{\a}(K_\delta)\}.$$
\end{proof}


\section{Forms of arbitrary degrees}

This section is devoted to proving our main result in its full generality, when the ideals $I_i$s are generated in arbitrary degrees.
We start by recalling the notion of a Stanley decomposition of multigraded modules.

\begin{definition} Let $\G$ be a finitely generated abelian group and let $B = A[T_1, \dots, T_r]$ be a $\G$-graded polynomial ring over a commutative ring $A$. Let $M$ be a finitely generated $\G$-graded $B$-module. A {\it Stanley decomposition} of $M$ is a finite decomposition  of the form
$$M = \bigoplus_{i=1}^m u_iA[Z_i],$$
where the direct sum is as $A$-modules, $u_i$s are $\G$-homogeneous elements in $M$, $Z_i$s are subsets (could be empty) of the variables $\{T_1, \dots, T_r\}$, and $u_iA[Z_i]$ denotes the $A$-submodule of $M$ generated by elements of the form $u_im$ where $m$ is a monomial in the polynomial ring $A[Z_i]$.
\end{definition}

The following lemma is well known in $\NN$-graded or standard $\NN^n$-graded
situations (cf. \cite{BG, BKU, HP}).

\begin{lemma} \label{lem.monomial}
Let $\G$ be a finitely generated abelian group and let $B = A[T_1, \dots, T_r]$ be a $\G$-graded polynomial ring. Let $I$ be a monomial ideal in $B$. Then a Stanley decomposition of $B/I$ exists.
\end{lemma}

\begin{proof} The proof follows along the same lines as in the proof of
\cite[Corollary 6.4]{HP} or \cite[Theorem 2.1]{BG}, as one notices that any
monomial is a homogeneous element.
\end{proof}

\begin{theorem} \label{thm.arbitrarydegree}
Let $\G$ be a finitely generated abelian group, $B = A[T_1, \dots, T_r]$ be a $\G$-graded polynomial ring over a commutative ring $A$ and  $\M$ be a finitely generated $\G$-graded $B$-module. 
Let $\Gamma$ denote the set of subsets of $\{\deg_\G(T_i)\}_{i=1}^r$ whose elements are linearly independent over $\ZZ$.  Then there exist a collection of pairs $(\delta_p ,E_p)\in \G\times \Gamma$, for $p = 1, \dots, m$, such that
$$\supp_\G(\M) = \bigcup_{p=1}^m \big( \delta_p + \langle E_p\rangle \big),$$
where $\langle E_p\rangle$ represents the free submonoid of $\G$ generated by elements in $E_p$.
\end{theorem}

\begin{proof} Since $\M$ is a finitely generated $\G$-graded $B$-module, there exists a homogeneous surjective map $\phi: F \twoheadrightarrow \M$ from a free $B$-module $F$ to $\M$. We can write $F = \bigoplus_{i=1}^mBe_i$, where $\deg_G(e_i)$ represents the degree of the $i$-th generator of $M$. Let $K = \operatorname{ker} \phi$. Then $\M \simeq F/K$. In particular, $\supp_\G(\M) = \supp_\G(F/K)$.

Extend any monomial order on $B$ to a monomial order on $F = \bigoplus_{i=1}^mBe_i$, by ordering the $e_i$'s.  Since $\M$ is $\G$-graded, so is $K$. Thus, by Proposition \ref{lem.initial}, we have $\supp_\G(F/K) = \supp_\G(F/\operatorname{in}_\prec(K))$. Therefore,
$$\supp_\G(\M) = \supp_\G(F/\operatorname{in}_\prec(K)).$$

Observe that $\operatorname{in}_\prec(K)$ is generated by monomials of the form $\T^{\alpha}e_i$ (where $\T^\alpha = T_1^{\alpha_1} \cdots T_r^{\alpha_r}$ for $\alpha = (\alpha_1, \dots, \alpha_r) \in \NN^r$). Let $I_i$ be the monomial ideal in $B$ generated by all monomials $\T^{\alpha}$ for which $\T^{\alpha}e_i \in \operatorname{in}_\prec(K)$. Clearly, $F/\operatorname{in}_\prec(K) \simeq \bigoplus_{i=1}^m \dfrac{B}{I_i}e_i.$

By Lemma \ref{lem.monomial}, for each $i = 1, \dots, m$, there exists a Stanley decomposition of $\dfrac{B}{I_i}e_i$
$$\dfrac{B}{I_i}e_i \simeq \bigoplus_{j=1}^{m_i} u_{ij}A[T_{ij}],$$
where $u_{ij}$ are homogeneous elements of $\dfrac{B}{I_i}e_i$ and $T_{ij}$ are subsets (could be empty) of the variables $\{T_1, \dots, T_r\}$ in $B$. This gives
$$F/\operatorname{in}_\prec(K) \simeq \bigoplus_{i=1}^m \bigoplus_{j=1}^{m_i} u_{ij}A[T_{ij}].$$
Thus, $\supp_\G(F/\operatorname{in}_\prec(K))$ can be written as a finite union of the form
$$\supp_\G(F/\operatorname{in}_\prec(K)) = \bigcup_{i,j} \supp_\G\big(u_{ij}A[T_{ij}]\big).$$
Let $\delta_{ij} = \deg_\G(u_{ij})$. Then
$$\supp_\G(F/\operatorname{in}_\prec(K)) = \bigcup_{i,j} \big(\delta_{ij} + \supp_\G(A[T_{ij}]) \big).$$
To prove the theorem, it now suffices to show that $\supp_\G(A[T_{ij}])$ can be decomposed into a union of free submonoids of $\G$ of the form $\langle E \rangle$, where $E$ is a linearly independent subset in $\Gamma$.

Since $A[T_{ij}]$ is a polynomial ring whose variables are variables of $B$, without loss of generality, we may assume that $T_{ij} = \{T_1, \dots, T_r\}$, i.e., $A[T_{ij}] = B$. Let $H$ be the binomial ideal in $B$ generated by $\{\T^\alpha - \T^\beta ~\big|~ \deg_\G(\T^\alpha) = \deg_\G(\T^\beta)\}$. Then taking the quotient $B/H$ is the same as identifying monomials of the same degree in $B$. Thus, we have
$$
\supp_\G(B) = \supp_\G(B/H)= \supp_\G(B/\operatorname{in}_\prec (H))
$$
and
\begin{align}
\Big(\dfrac{B}{H}\Big)_\gamma =
\Big(\dfrac{B}{\operatorname{in}_\prec (H)}\Big)_\gamma =\left\{
\begin{array}{cl}
A & \text{if } \gamma \in \supp_\G(B/H) \\
0 & \text{otherwise.}
\end{array} \right. \label{eq.HF}
\end{align}

By Lemma \ref{lem.monomial}, a Stanley decomposition of $B/\operatorname{in}_\prec (H)$ exists. That is, we can write
$$B/\operatorname{in}_\prec (H) = \bigoplus_{j=1}^s u_jA[Z_j]$$
where $u_j$ are $\G$-homogeneous elements of $B/\operatorname{in}_\prec (H)$ and $Z_j$ are subsets (could be empty) of the variables $\{T_1, \dots, T_r\}$. Let $E_j$ be the set of degrees of variables in $Z_j$. It further follows from (\ref{eq.HF}) that $B/\operatorname{in}_\prec (H)$ has at most one monomial in each degree. This implies that the support of $u_jA[Z_j]$ are all disjoint and each set $E_j$ is linearly independent. Hence, by letting $\sigma_j = \deg_\G(u_j)$, we have
$$\supp_\G(B) = \coprod_{j=1}^s \big(\sigma_j + \langle E_j \rangle \big).$$
The theorem is proved.
\end{proof}

\begin{remark} It would be nice if the union in Theorem \ref{thm.arbitrarydegree} is a disjoint union. However, this is not true. Let $B = A[x,y]$ be a $\ZZ$-graded polynomial ring with $\deg(x) = 4$ and $\deg(y) = 7$ (hence, $\Gamma = \{4,7\}$). Let $\M = B/(x) \oplus B/(y) \simeq A[y]\oplus A[x]$. Then $\supp_\ZZ(\M) = \{4a+7b ~\big|~ a,b \in \ZZ_{\ge 0}\}.$ Moreover, linearly independent subsets of $\Gamma$ are $\{4\}$ and $\{7\}$. It can be easily seen that $\supp_\ZZ(\M)$ cannot be written as disjoint union of shifted free submonoids of $\ZZ$ generated by $4$ and/or by $7$.
\end{remark}

For a vector $\c = (c_1, \dots, c_s) \in \ZZ^s$ and a tuple $E = (\nu_1, \dots, \nu_s)$ of elements in $G$, we shall denote $\Delta E$ the empty tuple if $s\leq 1$ and the $(s-1)$-tuple $(\nu_2 -\nu_1,\ldots ,\nu_s -\nu_{s-1})$ else,
and by $\c.E$ the $G$-degree $\sum_{j=1}^s c_j\nu_j$. 
If $E$ and $E'$ are tuples, we denote by $E\vert E'$ the concatenation of $E$ and $E'$.

\begin{remark}
With some simple linear algebra arguments, it can be seen that for tuples $E_1,\ldots ,E_s$ of elements of $G$, the tuple of elements of $G\times \ZZ^s$, $E_1\times \{ e_1\} \vert \cdots \vert E_s\times \{ e_s\}$, where $e_i$ is the $i$-th basis element of $\ZZ^s$, is linearly independent if and only if $\Delta E_{1}\vert \cdots \vert\Delta E_{s}$ is linearly independent. These equivalent conditions imply that for $(\c_1,\ldots ,\c_s)\not= (\c'_1,\ldots ,\c'_s)$, with $\c_i ,\c'_i \in \ZZ_{\ge 0}^{|E_{i}|}$ and $\vert \c_i \vert =\vert \c'_i \vert $ for all $i$, one has $\c_1.E_{1} +\cdots +\c_s.E_{s}\not= \c'_1.E_{1} +\cdots +\c'_s.E_{s}$. This last fact is a direct corollary of the linear independence of $E_1\times \{ e_1\} \vert \cdots \vert E_s\times \{ e_s\}$, as
$\c_1.(E_{1}\times \{ e_1\}) +\cdots +\c_s.(E_{s}\times \{ e_s\})=(\c_1.E_{1} +\cdots +\c_s.E_{s})\times (\vert \c_1 \vert e_1+\cdots +\vert \c_s \vert e_s)$.
\end{remark}

We shall now prove our main result. Recall that $I_i = (f_{i,1}, \dots, f_{i,r_i})$ and let $\gamma_{i,j} = \deg_G(f_{i,j})$.

\begin{theorem} \label{thm.arbitrarydegree2}
Let $G$ be a finitely generated abelian group and let $S=A[x_1, \dots, x_n]$ be a $G$-graded algebra over a commutative ring $A\subseteq S_0$. Let $I_i=(f_{i,1}, \dots, f_{i,r_i})$ for $i=1,\ldots s$ be  $G$-homogeneous ideals in $S$, and let $M$ be a finitely generated $G$-graded $S$-module. Set $\Gamma_i = \{\deg_G(f_{i,j})\}_{j=1}^{r_i}$. Let $\ell \ge 0$ and assume that $\ell =0$ or $A$ is Noetherian.

 There exist a finite collection of elements $\delta_p^\ell \in G$, a finite collection of integers $t_{p,i}^\ell$, and a finite collection of non-empty tuples $E_{p,i}^\ell \subseteq \Gamma_i$, such that $\Delta E_{p,1}^\ell \vert \cdots \vert\Delta E_{p,s}^\ell$ is linearly independent for all $p$, satisfying :

$$\supp_G(\tor^S_\ell(MI_1^{t_1}\cdots I_s^{t_s},A)) = \bigcup_{p=1}^m \big( \delta_p^\ell  +\bigcup_{\c_i \in \ZZ_{\ge 0}^{|E_{p,i}^\ell |}, |\c_i | = t_i-t^\ell_{p,i}} \c_1.E_{p,1}^\ell +\cdots +\c_s.E_{p,s}^\ell \big),
$$
 if $t_i\geq \max_p \{ t_{p,i}^\ell\}$ for all $i$.
\end{theorem}

\begin{proof} As before, we use $t$ to denote $(t_1, \dots, t_s) \in \ZZ^s$, and let $\R := \bigoplus_{t \ge 0}I^{t}T^{t}$ and $M\R := \bigoplus_{t \ge 0}MI^{t}T^{t}$. Consider $R = A[x_1, \dots, x_n][ T_{i,j}, 1\leq i\leq s,\; 1\leq j\leq r_i]$, the $G \times \ZZ^s$-graded polynomial ring over $A[x_1, \dots, x_n]$ with $\deg_{G \times \ZZ^s}(x_i) = (\deg_G(x_i), 0)$ and $\deg_{G \times \ZZ^s}(T_{i,j}) = (\deg_G(f_{i,j}),e_i)$, where $e_i$ denotes the $i$-th canonical generator of $\ZZ^s$. The natural surjective map $\phi: R \twoheadrightarrow \R$ that sends $x_i$ to $x_i$ and $T_{i,j}$ to $f_{i,j}T_i$ makes $M\R$ a finitely generated $G \times \ZZ^s$-graded module over $R$.

Let $\F_\bullet$ be a  $G \times \ZZ^s$-graded free resolution of $M\R$ over $R$. If $A$ is Noetherian, then each $\F_i$ can be chosen of finite rank, and we make such a choice. For $t\in\ZZ^s$, the degree $(*,t)$-strand $\F^t_\bullet$ of $\F_\bullet$ provides a $G$-graded free resolution of $MI^t$ over $S = R_{(*,0)}$. Thus,
$$\tor^S_i(MI^t,A) = H_i(\F_\bullet^t \otimes_S A).$$
Moreover, taking homology respects the graded structure, and therefore,
$$H_i(\F_\bullet^t \otimes_S A) = H_i(\F_\bullet \otimes_R R/\m R)_{(*,t)},$$
where $\m = (x_1, \dots, x_n)$ is the homogeneous irrelevant ideal in $S$.

Let $\Gamma' = \{(\gamma_{i,j},e_i) \in G \times \ZZ^s\}=\coprod_i \Gamma_i\times \{ e_i\}$ be the set of degrees of the variables $T_{i,j}$. Observe that $H_i(\F_\bullet \otimes_R R/\m R)$ is a finitely generated $G \times \ZZ^s$-graded module over $R/\m R \simeq B$ for any $i$ if $A$ is Noetherian, and for $i=0$ in any case. Applying Theorem \ref{thm.arbitrarydegree} to the $G\times \ZZ^s$-graded module $H_i(\F_\bullet \otimes_R R/\m R)$ we obtain a finite collection of elements $\theta_p^\ell \in G \times \ZZ^s$ and a finite collection of linearly independent subsets $E^\ell_p\subseteq \Gamma'$ (which we view in a fixed order as tuples), for $p = 1, \dots, m$, such that
$$\supp_{G \times \ZZ^s}\big(H_i(\F_\bullet \otimes_R R/\m R)\big) = \bigcup_{p=1}^m \big(\theta^\ell_p + \langle E^\ell_p\rangle\big).$$
Let $\theta_p^\ell = (\delta_p^\ell, t_{1,p}^\ell ,\ldots ,t_{s,p}^\ell )$, where $\delta_p^\ell \in G$ and $t_{i,p}^\ell \in \ZZ$.
One has $E_p^\ell =\coprod_{i=1}^{s}E_{p,i}^\ell \times \{e_i\}$. The linear independence of the elements in $E_p^\ell$ is equivalent to the fact that the elements of $\Delta E_{p,1}^\ell \vert \cdots \vert\Delta E_{p,s}^\ell$ are linearly independent.

Taking the degree $(*,t)$-strand of $H_i(\F_\bullet \otimes_R R/\m R)$, we get
$$\supp_G(\tor^S_\ell(MI_1^{t_1}\cdots I_s^{t_s},A)) = \bigcup_{p=1}^m \big( \delta_p^\ell  +\bigcup_{\c_i \in \ZZ_{\ge 0}^{|E_{p,i}^\ell |}, |\c_i | = t_i-t^\ell_{p,i}} \c_1.E_{p,1}^\ell +\cdots +\c_s.E_{p,s}^\ell \big).$$
\end{proof}

\begin{example} Let $I \subseteq S$ be a complete intersection ideal of three forms $f_1,f_2,f_3$ of degrees $a,b,c$. On easily sees, for instance by the Hilbert Burch theorem, that
$$ \xymatrix{ 0\ar[r]&{\begin{array}{c}R(-a-b-c,-2)\\\bigoplus\\ R(-a-b-c,-1)\\ \end{array}}\ar^(.54){\left( \begin{matrix}T_1&T_2&T_3\\f_1&f_2&f_3\\  \end{matrix}\right) }[rr]&& {\begin{array}{c}R(-b-c,-1)\\\bigoplus\\ R(-a-c,-1)\\\bigoplus\\ R(-a-b,-1)\\ \end{array}}\ar^(.67){\left( \begin{matrix}f_2T_3-f_3T_2\\ f_3T_1-f_1T_3\\f_1T_2-f_2T_1\\\end{matrix}\right) }[rr] &&R\ar[r]&\R_I\ar[r]&0\\ } $$
is a graded free $R$-resolution of $\R_I$.

It follows that $\tor_0^S(I^t,A)_\mu =B_{\mu ,t}$, $\tor_2^S(I^t,A)_\mu =B_{\mu -a-b-c ,t-1}$, and
$$ \tor_1^S(I^t,A)_\mu =\hbox{\rm coker} \left( \xymatrix{B_{\mu -a-b-c,t-2}\ar^(.54){\left( \begin{matrix}T_1&T_2&T_3\\  \end{matrix}\right) }[rr]&& {\begin{array}{c}B_{\mu-b-c,t-1}\\\bigoplus\\ B_{\mu -a-c,t-1}\\\bigoplus\\ B_{\mu -a-b,t-1}\\ \end{array}}} \right). $$
Now, $\tor_1^S(\R_I,A)$ has the following Stanley decomposition:
$$
A[T_1,T_2,T_3](-b-c,-1)\bigoplus A[T_1,T_3](-a-c,-1)\bigoplus A[T_1,T_2,T_3](-a-b,-1). $$
The ideal $H$ generated by the binomials $T^\alpha -T^\beta$ with $\deg (T^\alpha  )=\deg (T^\beta )$ is the kernel of the map
$$ \xymatrix@R=1pt{ A[T_1,T_2,T_3]\ar[r]& A[u,v]\\  T_1\ar[r]&uv^a\\  T_2\ar[r]&uv^b\\  T_3\ar[r]&uv^c\\  }  $$
and is therefore generated by a single irreducible and homogeneous binomial.

If, for example, $(a,b,c)=(2,5,8)$ then this relation is $T_2^2-T_1T_3$, and  a Stanley decomposition of, for instance, $A[T_1,T_2,T_3]/(T_2^2)$ is $A[T_1,T_3]\oplus T_2A[T_1,T_3](-5,-1)$. It finally gives the following  decomposition for $ \tor_1^S(\R_I ,A)$
:
$A[T_1,T_3](-13,-1)\oplus T_2A[T_1,T_3](-18,-2)\oplus A[T_1,T_3](-10,-1)\oplus A[T_1,T_3](-7,-1)\oplus T_2A[T_1,T_3](-12,-2).$

Setting $E_t:=\{ 2\alpha +8\beta ~\vert\ \alpha, \beta \in \ZZ_+, \alpha +\beta =t\}=2t+6\{ 0,\cdots ,t\}$, one gets that for $t\geq 2$,

 -- $\supp_{\ZZ}(\tor_0^S(I^t,A))=E_t \cup (5+E_{t-1})$,

  -- $\supp_{\ZZ}(\tor_1^S(I^t,A))=(13+E_{t-1}) \cup (18+E_{t-2}) \cup (10+E_{t-1}) \cup (7+E_{t-1}) \cup (12+E_{t-2})$,

   -- $\supp_{\ZZ}(\tor_2^S(I^t,A))=(15+E_{t-1}) \cup (20+E_{t-2})$.

   Notice that one has the simplified expression :

    -- $\supp_{\ZZ}(\tor_1^S(I^t,A))=(5+E_{t})  \cup (10+E_{t-1})$.

\end{example}


\end{document}